\documentclass[12pt,reqno]{amsart}
\usepackage{amsmath,amsfonts,amsthm,amssymb,amsxtra}
\newcommand\version{November 14, 2008}

\newtheorem{theorem}{Theorem}[section]
\newtheorem{proposition}[theorem]{Proposition}
\newtheorem{lemma}[theorem]{Lemma}
\newtheorem{corollary}[theorem]{Corollary}

\theoremstyle{definition}

\theoremstyle{remark}
\newtheorem{remark}[theorem]{Remark}
\newtheorem{assumption}[theorem]{Assumption}

\numberwithin{equation}{section}

\newcommand{\C}{\mathbb{C}}

\newcommand{\const}{{\rm const}\,}

\renewcommand{\epsilon}{\varepsilon}

\newcommand{\loc}{{\rm loc}}
\newcommand{\N}{\mathbb{N}}

\renewcommand{\phi}{\varphi}
\newcommand{\R}{\mathbb{R}}
\newcommand{\Sph}{\mathbb{S}}

\newcommand{\Z}{\mathbb{Z}}
\newcommand\sm{\setminus\{0\}}

\DeclareMathOperator{\Div}{div}

\DeclareMathOperator{\re}{Re}
\DeclareMathOperator{\spec}{spec}
\DeclareMathOperator{\supp}{supp}
\DeclareMathOperator{\sgn}{sgn}

\begin{document}

\title[Hardy inequality  --- \version]{Non-linear ground state representations \\and sharp Hardy inequalities}

\author{Rupert L. Frank}
\address{Rupert L. Frank, Department of Mathematics,
Princeton University, Washington Road, Princeton, NJ 08544, USA}
\email{rlfrank@math.princeton.edu}

\author{Robert Seiringer}
\address{Robert Seiringer, Department of Physics, Princeton University,
P.~O.~Box 708, Princeton, NJ 08544, USA}
\email{rseiring@princeton.edu}

\begin{abstract}
We determine the sharp constant in the Hardy inequality for fractional Sobolev spaces. To do so, we develop a non-linear and non-local version of the ground state representation, which even yields a remainder term. From the sharp Hardy inequality we deduce the sharp constant in a Sobolev embedding which is optimal in the Lorentz scale. In the appendix, we characterize the cases of equality in the rearrangement inequality in fractional Sobolev spaces.
\end{abstract}

%\date{\version}

\maketitle

%%%%%%%%%%%%%%%%%%%%%%%%%%%%%%%%%%%%%%%%%%%%%%%%%%%%%%%%%%%%%%%%%%%%%%

\section{Introduction and main results}

Hardy's inequality plays an important role in many questions from mathematical physics, spectral theory, analysis of linear and non-linear PDE, harmonic analysis and stochastic analysis. It states that
\begin{equation}\label{eq:hardyclass}
 \int_{\R^N} |\nabla u|^p \,dx \geq \left(\frac{|N-p|}p\right)^p \int_{\R^N} \frac{|u(x)|^p}{|x|^p} \,dx \,,
\end{equation}
and holds for all $u\in C_0^\infty(\R^N)$ if $1\leq p < N$, and for
all $u\in C_0^\infty(\R^N\setminus\{0\})$ if $p>N$. The constant on
the right side of \eqref{eq:hardyclass} is sharp and, for $p>1$, not
attained in the corresponding homogeneous Sobolev spaces $\dot
W^1_p(\R^N)$ and $\dot W^1_p(\R^N\setminus\{0\})$, respectively, i.e.,
the completion of $C_0^\infty(\R^N)$ and
$C_0^\infty(\R^N\setminus\{0\})$ with respect to the left side of
\eqref{eq:hardyclass}. If $p=1$, equality holds for any symmetric
decreasing function.

In this paper we are concerned with the fractional analog of Hardy's inequality \eqref{eq:hardyclass}, where the left side is replaced by
\begin{equation}\label{eq:energy}
 \iint_{\R^N\times\R^N} \frac{|u(x)-u(y)|^p}{|x-y|^{N+ps} } \,dx\,dy
\end{equation}
for some $0<s<1$. By scaling the function $|x|^{-p}$ on the right side has to be replaced by $|x|^{-ps}$. For $N\geq 1$ and $0<s<1$  we consider the homogeneous Sobolev spaces $\dot{W}^s_p(\R^N)$ and $\dot{W}^s_p(\R^N\setminus\{0\})$ defined as the completion with respect to \eqref{eq:energy} of $C_0^\infty(\R^N)$ for $1\leq p < N/s$ and $C_0^\infty(\R^N\sm)$ for $p>N/s$, respectively. Our main result is the optimal constant in the fractional Hardy inequality.

\begin{theorem}[\textbf{Sharp fractional Hardy inequality}]\label{main}
 Let $N\geq 1$ and $0<s<1$.  Then for all $u\in \dot{W}^s_p(\R^N)$ in case $1\leq p < N/s$, and for all $u\in \dot{W}^s_p(\R^N\sm)$ in case $p > N/s$,
\begin{equation}\label{eq:main}
 \iint_{\R^N\times\R^N} \frac{|u(x)-u(y)|^p}{|x-y|^{N+ps} } \,dx\,dy
\geq \mathcal C_{N,s,p} \int_{\R^N} \frac{|u(x)|^p}{|x|^{ps}}\,dx
\end{equation}
with
\begin{equation}\label{eq:mainconst}
\mathcal C_{N,s,p} :=
2 \int_0^1 r^{ps-1} \left|1 - r^{(N-ps)/p} \right|^p \Phi_{N,s,p}(r) \, dr\,,
\end{equation}
and
\begin{equation}\label{eq:mainphi}
\begin{split}
\Phi_{N,s,p}(r) := |\Sph^{N-2}| \int_{-1}^1 \frac{\left(1-t^2 \right)^{\tfrac{N-3}2} \,dt}{(1-2rt+r^2)^{\tfrac{N+ps}2}} \,,
\quad N\geq 2 \,, \\
\Phi_{1,s,p}(r) := \left( \frac{1}{(1-r)^{1+ps}} + \frac{1}{(1+r)^{1+ps}} \right) \,,
\quad N=1\,.
\end{split}
\end{equation}
The constant $\mathcal C_{N,s,p}$ is optimal. If $p=1$, equality holds
iff $u$ is proportional to a symmetric-decreasing function. If $p>1$,
the inequality is strict for any function $0\not\equiv u\in
\dot{W}^s_p(\R^N)$ or $\dot{W}^s_p(\R^N\sm)$, respectively.
\end{theorem}

For $p=1$ and, e.g., $N=1$ or $N=3$ one finds
$$
\mathcal C_{1,s,1} = \frac{2^{2-s}}{s} \,, \quad
\mathcal C_{3,s,1} = 4\pi \frac{2^{1-s}}{s(s-1)} \,.
$$
For general values of $p$ and $N$ the double integral is easily evaluated numerically or estimated analytically (see also \eqref{eq:constant1d} and \eqref{eq:constantn} below for different expressions).
For $p=2$ one can evaluate $\mathcal C_{N,s,p}$ via Fourier transform \cite{FLS} and obtains the well-known expression 
\begin{equation}
 \label{eq:mainconst2}
\mathcal C_{N,s,2} = 2 \pi^{N/2} \frac{\Gamma((N+2s)/4)^2}{\Gamma((N-2s)/4)^2} \frac{|\Gamma(-s)|}{\Gamma((N+2s)/2)} \,. 
\footnote{This corrects a typographical error in the published version of this paper. We thank C. A. Sloane for pointing this out to us.}
\end{equation}
This was first derived by Herbst \cite{H}; see also \cite{KPS,B,Y} for different proofs. 
Indeed, Herbst determined the sharp constants in the inequality
\begin{equation}\label{eq:herbst}
 \left\| (-\Delta)^{s/2} u \right\|_p^p 
\geq \tilde{\mathcal C}_{N,s,p} \int_{\R^N} \frac{|u(x)|^p}{|x|^{ps}}\,dx
\end{equation}
for arbitrary $1<p<N/s$. For $p=2$ the left side is well-known to be proportional to the left side in \eqref{eq:main}. For $p\neq 2$ and $0<s<1$, however, the expression on the left side is \emph{not} equivalent to \eqref{eq:energy}. There is a one-sided inequality according to whether $1<p<2$ or $p>2$; see, e.g., \cite[Ch. V]{S}. In particular, the sharp constant $\tilde{\mathcal C}_{N,s,1}$ in \eqref{eq:herbst} for $p=1$ is zero, as opposed to \eqref{eq:main}.

\bigskip

One of our motivations is the recent work by Bourgain, Brezis and Mironescu \cite{BBM1, BBM2} and by Maz'ya and Shaposhnikova \cite{MS}. Consider the case $N>ps$, and recall that the Sobolev embedding theorem asserts that $\dot{W}^s_p(\R^N) \subset L_{p^*}(\R^N)$ for $p^*=Np/(N-ps)$ with
\begin{equation}\label{eq:sobemb}
 \iint_{\R^N\times\R^N} \frac{|u(x)-u(y)|^p}{|x-y|^{N+ps} } \,dx\,dy \geq \mathcal S_{N,s,p} \|u\|_{p^*}^p \,,
\end{equation}
see, e.g., \cite[Thms. 7.34, 7.47]{Ad}. The optimal values of the constants $\mathcal S_{N,s,p}$ are unknown. In \cite{BBM2} Bourgain, Brezis and Mironescu obtained quantitative estimates on the constants $\mathcal S_{N,s,p}$ which reflect the correct behavior in the limits $s\to 1$ or $p\to N/s$. (More precisely, these authors studied the corresponding problem for functions on a cube with zero average, but this problem is equivalent to the problem on the whole space, see \cite[Rem.~1]{BBM2} or \cite[Cor.~1]{MS}.) The proof in \cite{BBM2} relies on advanced tools from harmonic analysis. It was simplified and extended by Maz'ya and Shaposhnikova \cite{MS} who showed that the sharp constant in \eqref{eq:sobemb} satisfies
\begin{equation}\label{eq:bbm}
\mathcal S_{N,s,p} \geq c(N,p) \, \frac{(N-ps)^{p-1}}{s(1-s)} \,.
\end{equation}
The key observation in \cite{MS} was that \eqref{eq:bbm} follows from a sufficiently good bound on the constant in the fractional Hardy inequality. Maz'ya and Shaposhnikova did not, however, determine the optimal constants in this inequality. Their bound
\begin{equation}\label{eq:ms}
\mathcal C_{N,s,p} \geq \tilde c(N,p) \, \frac{(N-ps)^p}{s(1-s)} \,,
\end{equation}
which leads to the Bourgain--Brezis--Mironescu result \eqref{eq:bbm}, is easily recovered from our explicit expression for $\mathcal C_{N,s,p}$.

In fact, in Section \ref{sec:embedding} below we show that our sharp Hardy inequality implies an even stronger result. Namely, together with a symmetrization argument it yields a simple proof of the embedding
\begin{equation}\label{eq:embeddinglor}
\dot W^s_p(\R^N) \subset L_{p^*,p}(\R^N),
\quad 1\leq p <N/s\,,\ p^*=Np/(N-ps)\,,
\end{equation}
due to Peetre \cite{P}. Here $L_{p^*,p}(\R^N)$ denotes the Lorentz space, the definition of which is recalled in Section~\ref{sec:embedding}. Embedding \eqref{eq:embeddinglor} is optimal in the Lorentz scale. Since $L_{p^*,p}(\R^N)\subset L_{p^*}(\R^N)$ with strict inclusion, \eqref{eq:embeddinglor} is stronger than \eqref{eq:sobemb}. While we know only of non-sharp proofs of \eqref{eq:embeddinglor} via interpolation theory, our Theorem \ref{embedding} below gives the optimal constant in this embedding and characterizes all optimizers. To do so, we need to characterize the optimizers in the rearrangement inequality by Almgren and Lieb for the functional \eqref{eq:energy}, see Theorem \ref{rearrstrict}. For another recent application of Lorentz norms in connection with Hardy-Sobolev inequalities we refer to \cite{MS2}.

\bigskip

In contrast to the case $p=2$, there seems to be no way to prove
\eqref{eq:main} via Fourier transform if $p\neq 2$. Instead, our proof
is based on the observation that $|x|^{-(N-ps)/p}$ is a positive
solution of the Euler-Lagrange equation associated with
\eqref{eq:main} (but fails to lie in $\dot W^s_p(\R^N)$ or $\dot
W^s_p(\R^N\sm)$, respectively). Writing $u= |x|^{-(N-ps)/p} v$,
\eqref{eq:main} becomes an inequality for the unknown function
$v$. While it is well-known and straightforward to prove
\eqref{eq:hardyclass} in this way, this approach seems to be new in
the fractional case.

One virtue of our approach is that it automatically yields remainder terms. In particular, for $p\geq 2$ we obtain the following strengthening of \eqref{eq:main}.

\begin{theorem}[\textbf{Sharp Hardy inequality with remainder}]\label{remainder}
Let $N\geq 1$, $0<s<1$ and $p\geq 2$.  Then for all $u\in \dot{W}^s_p(\R^N)$ in case $p < N/s$, and for all $u\in \dot{W}^s_p(\R^N\sm)$ in case $p > N/s$, and $v=|x|^{(N-ps)/p} u$,
\begin{align}\label{eq:remainder}
\iint_{\R^N\times\R^N} \frac{|u(x)-u(y)|^p}{|x-y|^{N+ps} } \,dx\,dy
& - \mathcal C_{N,s,p} \int_{\R^N} \frac{|u(x)|^p}{|x|^{ps}}\,dx \notag \\
& \geq c_p\, \iint_{\R^N\times\R^N} \frac{|v(x)-v(y)|^p}{|x-y|^{N+ps} } \frac{dx}{|x|^{(N-ps)/2}} \frac{dy}{|y|^{(N-ps)/2}}
\end{align}
where $\mathcal C_{N,s,p}$ is given by \eqref{eq:mainconst} and $0<c_p\leq 1$ is given by
\begin{equation}\label{eq:gsrconst}
c_p:=\min_{0<\tau<1/2} \left((1-\tau)^p -\tau^p + p\tau^{p-1}\right)\,.
\end{equation}
If $p=2$, then \eqref{eq:remainder} is an equality with $c_2=1$.
\end{theorem}

We refer to the substitution of $u$ by $v= \omega^{-1} u$, where $\omega$ is a positive solution of the Euler-Lagrange equation of the functional under consideration, as `ground state substitution'. In the \emph{linear} and \emph{local} case, such representations go back at least to Jacobi and have numerous applications, among others, in the spectral theory of Laplace and Schr\"odinger operators (see the classical references \cite{Bi,He} and also \cite{D}), constructive quantum field theory (especially in the work by Segal, Nelson, Gross, and Glimm-Jaffe; see, e.g., \cite{GJ}) and Allegretto-Piepenbrink theory (developed in particular by Allegretto, Piepenbrink and Agmon; see, e.g., \cite{M,Pi} for references). Our goal in this paper is to derive a \emph{non-local} and \emph{non-linear} analog of such a representation. Despite all these applications, even in the linear case a non-local version of the ground state representation has only recently been found \cite{FLS}. While we were only interested in a special case in \cite{FLS}, here we wish to show that this formula holds in a much more general setting. Moreover, for $p> 2$ we will find a non-linear analog of this representation formula in the form of an inequality. This is the topic of Section \ref{sec:general} where we consider functionals of the form \eqref{eq:energy} with $|x-y|^{-N-ps}$ replaced by an arbitrary symmetric and non-negative, but not necessarily translation invariant kernel.

This paper is organized as follows. In Section \ref{sec:general} we derive Hardy inequalities and ground state representations in a general setting and in Section \ref{sec:proof} we apply this method to prove Theorems \ref{main} and \ref{remainder}. In Section \ref{sec:embedding} we show that Theorem \ref{main} implies the optimal Sobolev embedding \eqref{eq:embeddinglor} by using some facts from Appendix \ref{sec:rearr} about rearrangement in fractional Sobolev spaces.

\bigskip
{\em Acknowledgments.}
The authors wish to thank E. Lieb for helpful discussions. This work was supported by DAAD grant D/06/49117 (R.F.), by U.S. National Science Foundation grant PHY 06 52356 and an A.P. Sloan Fellowship (R.S.).

%%%%%%%%%%%%%%%%%%%%%%%%%%%%%%%%%%%%%%%%%%%%%%%%%%%%%%%%%%%%%%%%%%%%%%%%%%%%%%%%%%%%

\section{Ground state substitution}\label{sec:general}

\subsection{General Hardy inequalities}

We fix $N\geq 1$, $p\geq 1$ and a non-negative measurable function $k$ on $\R^N\times\R^N$ satisfying $k(x,y)=k(y,x)$ for all $x,y\in\R^N$. Our goal in this section is to provide a condition under which a Hardy inequality for the functional
$$
E[u] := \iint_{\R^N\times\R^N} |u(x)-u(y)|^p k(x,y) \,dx\,dy \ .
$$
holds. Loosely speaking, our assumption is that there exists a positive function $\omega$ satisfying the equation
\begin{equation}\label{eq:el}
 2\ \int_{\R^N} \left(\omega(x)-\omega(y)\right) \left|\omega(x)-\omega(y)\right|^{p-2} k(x,y) \,dy
= V(x) \omega(x)^{p-1}
\end{equation}
for some real-valued function $V$ on $\R^N$. We emphasize that if $k$ is too singular on the diagonal (for instance, in our case of primary interest $k(x,y)=|x-y|^{-N-ps}$, $s>0$) the integral on the left side will not be convergent and some regularization of principal value type will be needed. We think of $\omega$ as the `virtual ground state' corresponding to the energy functional $E[u] - \int V|u|^p\,dx$.

We formulate the precise meaning of \eqref{eq:el} as

\begin{assumption}\label{ass:el}
 Let $\omega$ be a positive, measurable function on $\R^N$. There exists a family of measurable functions $k_\epsilon$, $\epsilon>0$, on $\R^N\times\R^N$ satisfying $k_\epsilon(x,y)=k_\epsilon(y,x)$, $0\leq k_\epsilon(x,y) \leq k(x,y)$ and
\begin{equation}\label{eq:assk}
\lim_{\epsilon\to 0} k_\epsilon(x,y) = k(x,y)
\end{equation}
for a.e. $x,y\in\R^N$. Moreover, the integrals
\begin{equation}\label{eq:elreg}
 V_\epsilon(x) := 2\ \omega(x)^{-p+1} \int_{\R^N} \left(\omega(x)-\omega(y)\right) \left|\omega(x)-\omega(y)\right|^{p-2} k_\epsilon(x,y) \,dy
\end{equation}
are absolutely convergent for a.e. $x$, belong to $L_{1,\loc}(\R^N)$ and $V:=\lim_{\epsilon\to 0} V_\epsilon$ exists weakly in $L_{1,\loc}(\R^N)$, i.e., $\int V_\epsilon g\,dx \to \int V g\,dx$ for any bounded $g$ with compact support.
\end{assumption}

The following is a general version of Hardy's inequality.

\begin{proposition}\label{hardy}
Under Assumption \ref{ass:el}, for any $u$ with compact support and $E[u]$ and $\int V_+ |u|^p \,dx$ finite one has
\begin{equation}\label{eq:hardy}
 E[u] \geq \int_{\R^N} V(x) |u(x)|^p \,dx \ .
\end{equation}
\end{proposition}

In applications where additional properties of $k$ and $V$ are
available, the assumption that $u$ has compact support can typically
be removed by some limiting argument. It appears here because we want
to work with the rather minimal Assumption~\ref{ass:el}.

Our next result improves this in the case $p\geq 2$ by giving an explicit remainder estimate. It involves the functional
$$
E_\omega[v] := \iint_{\R^N\times\R^N} |v(x)-v(y)|^p \, \omega(x)^{\tfrac p2} k(x,y) \omega(x)^{\tfrac p2} \,dx\,dy
$$
and is a non-linear analog of what is known as `ground state representation formula'.

\begin{proposition}\label{gsr}
Let $p\geq 2$. Under Assumption \ref{ass:el}, for any $u$ with compact support write $u=\omega v$ and assume that $E[u]$, $\int V_+ |u|^p \,dx$, and $E_\omega[v]$ are finite. Then
\begin{equation}\label{eq:gsr}
 E[u] - \int_{\R^N} V(x) |u(x)|^p \,dx \geq c_p \, E_\omega[v]
\end{equation}
with $c_p$ from \eqref{eq:gsrconst}. If $p=2$, then \eqref{eq:gsr} is an equality with $c_2=1$.
\end{proposition}

We shall prove Propositions \ref{hardy} and \ref{gsr} in Subsection \ref{sec:hardy} after having discussed a typical application and having explained their analogs involving derivatives.

In this paper we are mostly interested in the case where $k(x,y)=|x-y|^{-N-ps}$ which enters in \eqref{eq:main}. For this particular choice of the kernel and for $p=2$, ground state representation \eqref{eq:gsr} (with equality) was proved in \cite{FLS}. The results for general kernels $k$ seem to be new, even in the linear case $p=2$.

\begin{remark}
 In the proofs of Propositions \ref{hardy} and \ref{gsr} we will not use that the underlying space is $\R^N$ or that the measure is Lebesgue measure. Hence similar results hold, e.g., when $\R^N$ is replaced by a domain $\Omega$. Another case of interest is that of the Laplacian on a weighted graph, where $\R^N$ is replaced by a (discrete) graph $\Gamma$ and $dx$ by the counting measure on $\Gamma$ and $E[u]$ is replaced by $\sum_{i,j\in\Gamma} k(i,j) |u_i-u_j|^p$ for a sequence $(u_i)_{i\in\Gamma}$. Propositions \ref{hardy} and \ref{gsr} continue to hold in this situation after the obvious changes. In the special case $p=2$, $\Gamma=\Z^N$ and $k$ such that $k(i,j)=0$ if $|i-j|>1$, one recovers a formula for Jacobi matrices which was recently proved in \cite{FSW}.
\end{remark}

In the special case $p=2$ and $k(x,y) = |x-y|^{-N-2s}$, the
representation (\ref{eq:el}) with {\em non-negative} $V$ gives a
simple sufficient condition for $V$ to be a multiplier from $\dot
W^s_2(\R^N)$ to $\dot W^{-s}_2(\R^N)$. For $s=1/2$ and general, not
necessarily sign-definite $V$ this problem is addressed in \cite{MV}.

%%%%%%%%%%%%%%%%%%%%%%%%%%%%%%%%%%%%%%%%%%%%%%%%%%%%%%%%%%%%%%%%

\subsection{Example}

A typical application of the ground state representation \eqref{eq:gsr} in mathematical physics concerns pseudo-relativistic Schr\"odinger operators $\sqrt{-\Delta +m^2} + V_0$ with a constant $m\geq 0$. Indeed, the kinetic energy can be put into the form considered in this section,
$$
\int \sqrt{|\xi|^2+m^2} |\hat u(\xi)|^2 \,d\xi = \iint |u(x)-u(y)|^2 k_m(|x-y|) \,dx\,dy
$$
where $\hat u(\xi)= (2\pi)^{-N/2} \int_{\R^N} e^{-i\xi\cdot x} u(x)\,dx$ is the Fourier transform of $u$ and
$$
k_m(r) = 
\begin{cases} 
 \left(\frac{m}{2\pi}\right)^{(N+1)/2} r^{-(N+1)/2} K_{(N+1)/2}(mr) & \textrm{if} \ m>0\,, \\
\pi^{-(N+1)/2} 2^{-1} \, \Gamma((N+1)/2) \, r^{-N-1} & \textrm{if} \ m=0\,,
\end{cases}
$$
with $K_\nu$ a Bessel function; see \cite[Sect. 7.11]{LL}.

More generally, one can consider non-negative functions $t$ and $k$ on $\R^N$ related by
\begin{equation}\label{eq:tk}
 t(\xi) = 4 \int_{\R^N} k(x) \sin^2(\xi\cdot x/2) \,dx
\end{equation}
and introduce the self-adjoint operator $T=t(D)$, $D=-i\nabla$, in $L_2(\R^N)$ with quadratic form
\begin{equation}
E[u] := \int_{\R^N} t(\xi) |\hat u(\xi)|^2 \,d\xi = \iint_{\R^N\times\R^N} |u(x)-u(y)|^2 k(x-y) \,dx\,dy \,.
\end{equation}
The last identity is a consequence of Plancherel's identity and \eqref{eq:tk}. We assume that $t$ is locally bounded and satisfies $t(\xi)\leq \const |\xi|^{2s}$ for some $0<s<1$ and all large $\xi$ and, similarly, that $k(x)$ is bounded away from the origin and satisfies $k(x)\leq \const |x|^{-N-2s}$ for all small $x$. Under these assumptions, $H^s(\R^N)=W^s_2(\R^N)$ is contained in the form domain of $T$ and we can consider the Schr\"odinger-type operator $T+V_0$ with a real-valued function $V_0\in L_{d/(2s)}(\R^N) + L_\infty(\R^N)$. Put $\lambda_0=\inf\spec\left(T+V_0\right)$ and assume that a positive function $\omega$ satisfies
$$
\left(T+V_0\right)\omega = \lambda_0 \omega
$$
in the sense of distributions. (Note that we do not require $\lambda_0$ to be an eigenvalue and $\omega$ an eigenfunction.) If $\omega$ is H\"older continuous with exponent $s$, then one easily verifies Assumption \ref{ass:el} and one obtains the ground state representation
\begin{equation}\label{eq:gsrschr}
\begin{split}
& \int_{\R^N} t(\xi) |\hat u(\xi)|^2 \,d\xi + \int_{\R^N} V_0(x) |u(x)|^2 \,dx - \lambda_0 \int_{\R^N} |u(x)|^2 \,dx \\
& \qquad = \iint |v(x)-v(y)|^2 \omega(x) k(x-y) \omega(y) \,dx\,dy
\end{split}
\end{equation}
for all $u$ in the form domain of $T$ and $v=\omega^{-1} u$.

%%%%%%%%%%%%%%%%%%%%%%%%%%%%%%%%%%%%%%%%%%%%%%%%%%%%%%%%%%%%%%%%%%%%%%%%%%

\subsection{The local case}

Before proving Propositions \ref{hardy} and \ref{gsr} we would like to recall their `local' analogs. Since these facts are essentially well known we shall ignore some technical details. Let $g$ be a positive function on $\R^N$ and put
$$
\tilde E[u] := \int_{\R^N} g |\nabla u|^p \,dx
$$
(with the convention that this is infinite if $u$ does not have a distributional derivative or if this derivative is not in $L_p(\R^N,g)$). Moreover, assume that $\omega$ is a positive weak solution of the weighted $p$-Laplace equation
\begin{equation}\label{eq:elloc}
-\Div( g|\nabla \omega|^{p-2} \nabla \omega ) = V \omega^{p-1} \ . 
\end{equation}
We claim that for any $u$ with $\tilde E[u]$ and $\int V_+ |u|^p \,dx$ finite one has
\begin{equation}\label{eq:hardyloc}
 \tilde E[u] \geq \int_{\R^N} V(x) |u(x)|^p \,dx \ .
\end{equation}
This is clearly the analog of \eqref{eq:hardy}. To prove \eqref{eq:hardyloc} we write $u=\omega v$ and use the elementary convexity inequality
\begin{equation}\label{eq:convexloc}
 |a+b|^p \geq |a|^p + p |a|^{p-2} \re \overline{a}\cdot b
\end{equation}
for vectors $a,b\in\C^N$ and $p\geq 1$. This yields
\begin{align*}
 \tilde E[u] & = \int_{\R^N} g | v \nabla \omega + \omega \nabla v|^p \,dx \\
& \geq  \int_{\R^N} g |v|^p  |\nabla \omega |^p \,dx + p \int_{\R^N} g |\nabla \omega |^{p-2} \omega \re \overline{v} |v|^{p-2} \nabla v \cdot \nabla\omega \,dx \ .
\end{align*}
Recognizing the integrand in the last integral as $p^{-1} g \omega |\nabla \omega|^{p-2}\nabla\omega\cdot \nabla( |v|^p )$ and integrating by parts using \eqref{eq:elloc} we arrive at \eqref{eq:hardyloc}.

Next we show that for $p\geq 2$, \eqref{eq:hardyloc} can be improved to
\begin{equation}\label{eq:gsrloc}
 \tilde E[u] - \int_{\R^N} V(x) |u(x)|^p \,dx 
\geq c_p \int g \omega^p |\nabla v|^p \,dx =: c_p \, \tilde E_\omega[v] \ .
\end{equation}
for $u=\omega v$ with $\tilde E[u]$, $\int V_+ |u|^p \,dx$, and $\tilde E_\omega[v]$ finite. This follows by the same argument as before if one uses instead of \eqref{eq:convexloc} its improvement
\begin{equation}\label{eq:convexloc2}
 |a+b|^p \geq |a|^p + p |a|^{p-2} \re \overline{a} \cdot b + c_p |b|^p
\end{equation}
for $p\geq 2$. One can show that $c_p$ given in \eqref{eq:gsrconst} is the sharp constant in this inequality.

Since \eqref{eq:convexloc2} is an equality for $p=2$ and $c_2=1$, so is \eqref{eq:gsrloc}. This is the ground state representation which is familiar from the spectral theory of differential operators. In the case $p\geq 2$, \eqref{eq:gsrloc} can be used to derive remainder terms in Hardy's inequality on domains; see, e.g., \cite{BFT}.

\begin{remark}
In the case $g\equiv 1$, $N\neq p$, and with $\omega(x)=|x|^{-(N-p)/p}$ and $v(x) = |x|^{(N-p)/p} u(x)$, the local Hardy inequality with remainder term yields the following improvement of \eqref{eq:hardyclass},
\begin{equation}\label{eq:lochardyrem}
 \int_{\R^N} |\nabla u|^p \, dx \geq \left( \frac {|N-p|}p\right)^p \int_{\R^N} \frac{ |u(x)|^p}{|x|^p} \, dx + c_p \int_{\R^N} |\nabla v|^p \frac{dx}{|x|^{N-p}} \, .
\end{equation}
The constant $c_p$ in (\ref{eq:lochardyrem}) is sharp for any $p\geq 2$. For $N>p$, this can 
be shown by using a trial function of the form $u(x)=|x|^{-(N-p)/p+\alpha}$ for
$|x|\leq 1$ and $u(x)=|x|^{-(N-p)/p-\epsilon}$ for $|x|\geq 1$, letting $\epsilon \to 0$ and
choosing $\alpha = (N-p)/(p\tau)$ where $0<\tau<1/2$ is the minimizer in
(\ref{eq:gsrconst}). Similarly, for $N<p$, we choose $u(x)=|x|^{-(N-p)/p+\alpha}$ for
$|x|\geq 1$ and $u(x)=|x|^{-(N-p)/p+\epsilon}$ for $|x|\leq 1$.
\end{remark}

%%%%%%%%%%%%%%%%%%%%%%%%%%%%%%%%%%%%%%%%%%%%%%%%%%%%%%%%%%%%%%%%%%%%%%%%%%%%%%%

\subsection{Proof of Propositions \ref{hardy} and \ref{gsr}}\label{sec:hardy}

We shall need the elementary

\begin{lemma}\label{numbers}
Let $p\geq 1$. Then for all $0\leq t\leq 1$ and $a\in \C$ one has
\begin{equation}\label{eq:numbers}
 |a-t|^p \geq (1-t)^{p-1} (|a|^p-t) \ .
\end{equation}
For $p>1$ this inequality is strict unless $a=1$ or $t=0$. Moreover, if $p\geq 2$ then for all $0\leq t\leq 1$ and all $a\in\C$ one has
\begin{equation}\label{eq:numbersimproved}
 |a-t|^p \geq (1-t)^{p-1} \left(|a|^p-t\right) + c_p \, t^{p/2} \, |a-1|^p \ ,
\end{equation}
with $0<c_p\leq 1$ given by \eqref{eq:gsrconst}. For $p=2$, \eqref{eq:numbersimproved} is an equality with $c_2=1$. For $p>2$, \eqref{eq:numbersimproved} is a strict inequality unless $a=1$ or $t=0$.
\end{lemma}

\begin{remark}
 The fact that in \eqref{eq:numbersimproved} the same constant $c_p$ as in \eqref{eq:convexloc2} appears is not a coincidence. Indeed, putting $a=1+\epsilon \tilde a$ and $t=1-\epsilon \tilde b$ for some $\tilde a\in\C$ and $\tilde b>0$ and expanding (\ref{eq:numbersimproved}) up to order $\epsilon^p$ we recover inequality \eqref{eq:convexloc2} with vectors $a$, $b$ replaced by numbers $\tilde a$,~$\tilde b$.
\end{remark}

\begin{proof}
  To prove the first assertion note that for fixed $|a|$ the minimum
  of the left side is clearly achieved for $a$ real and positive. Since
  for $|a|^p<t$ the inequality is trivial, one may thus assume that
  $a\geq t^{1/p}$. The assertion then follows from the fact that the
  derivative with respect to $a$ of $(a-t)^p/(a^p-t)$ vanishes only at
  $a=1$.

To prove the second assertion, we may assume that $p>2$, since (\ref{eq:numbersimproved}) is an equality if $p=2$. We first prove the assertion for real $a$. The function
$$ 
f(a,t) := \frac{|a-t|^p - (1-t)^{p-1} (|a|^p-t) } { t^{p/2}|a-1|^p}\,.
$$
diverges at $a=1$, and its partial derivative with respect to $a$ is given by
$$
 \frac{\partial f}{\partial a}(a,t) = \frac{p(1-t)^{p-2}}{t^{p/2} (a-1) |a-1|^p} \left( \frac{|a-t|^{p-2}(t-a)}{(1-t)^{p-1}} + \frac{ |a|^{p-2}a-t}{1-t} \right) \ .
$$
For $a> 1> t$ this is negative,  as follows from the first assertion with $p$ replaced by $p-1$. 
Hence for all $a >1$,
\begin{equation*}\label{eq:largepproof}
f(a,t) \geq f(+\infty,t)= t^{-p/2}\left( 1-(1-t)^{p-1} \right) \ .
\end{equation*}
An elementary calculation shows that the latter function is decreasing for $t\in(0,1)$. This proves that $f(a,t)\geq 1$ for $a > 1$.

Next, we claim that $f$ does not attain its minimum in the interior of the region $\{(a,t):\ -\infty <a<1,\  0<t<1\}$. To see this, we write the partial derivative of $f$ with respect to $t$ as 
\begin{align}\nonumber
\frac{\partial f}{\partial t}(a,t) &= \frac{(1-t)^{p-1}}{2 t^{(p+2)/2} |a-1|^p} \biggl( p(t+a) \left(\frac{|a-t|^{p-2}(t-a)}{(1-t)^{p-1}} + \frac{ |a|^{p-2}a-t}{1-t}\right) \\ \nonumber & \qquad\qquad\qquad\qquad  + \frac t{1-t} \left( (|a|^p-1)(p-2) -ap(|a|^{p-2}-1)\right)\biggl) \ . 
\end{align}
The first line vanishes in case $\partial f/\partial a=0$. Moreover, it is easy to see that the second line is non-zero for $a\in (-\infty,1)\setminus\{-1\}$. In fact, it is positive if $a\in(-\infty,-1)$ and negative if $a\in(-1,0]$. If $0<a<1$, it is negative in view of
$$
\frac{ a^p -1}{a^{p-1}-a} > \frac{p}{p-2} \ .
$$
(The latter inequality holds since the left side is strictly monotone decreasing.) To treat $a=-1$ one checks that $\partial f/\partial a\,(-1,t)\neq0$ for $0<t<1$. This proves that $f$ does not attain its minimum in the interior of the region $\{(a,t):\ -\infty <a<1,\, 0<t<1\}$.

Now we examine $f$ on the boundary of that region. Similarly as above, we have $\lim_{a\to -\infty} f(a,t) \geq 1$ uniformly in $t\in (0,1)$. Moreover, $\lim_{t\to 0} f(a,t)=+\infty$ uniformly in $a<1$, and $\lim_{t\to 1} f(a,t) =1$ uniformly in $a\leq 1-\epsilon$ for all $\epsilon>0$. Finally, $\lim_{a\to 1} f(a,t)= +\infty$ uniformly in $t\in (0,1-\epsilon)$ for all $0<\epsilon<1$. Thus it remains to study the limit $a\to 1$ and $t\to 1$. 
For given $\tau>0$ we let $a\to 1$ and $t\to1$ simultaneously with $1-t=\tau(1-a)$ and find
$$
\lim_{a\to 1} f(a,1-\tau(1-a))= |1-\tau|^p -\tau^p + p \tau^{p-1} \geq c_p\, .
$$
The last inequality follows from the definition of $c_p$ and the fact that the minimum over $\tau$ is attained for $\tau\in (0,1/2)$. This proves that $f(a,t)> c_p$ for all $a\in \R \setminus \{1\}$ and $0<t<1$.

Finally, we assume that $a$ is an arbitrary complex number. We write $a-t=x+iy$ with $x$ and $y$ real and put $\beta:=|a-t|$. What we want to prove is that for all $\beta\geq 0$ and $x\in(-\beta,\beta)$ one has
$$
(1-t)^{p-1} (\beta^2+2tx+t^2)^{p/2} + c_p t^{p/2} \left(\beta^2 -2(1-t)x +(1-t)^2 \right)^{p/2}
\leq \beta^p + (1-t)^{p-1} t \ .
$$
But for fixed $\beta$, the left side is a convex function of $x$ in the interval $(-\beta,\beta)$, so its maximum will be attained either at $x=\beta$ or $x=-\beta$, that is, for real values of $a-t$. This reduces the assertion in the complex case to the real case and completes the proof.
\end{proof}

We now turn to the

\begin{proof}[Proof of Proposition \ref{hardy}]
We may assume that $\int V_-|u|^p \,dx<\infty$, for otherwise there is nothing to prove. Replacing $u$ by $u\min\{1,M |u|^{-1}\}$ and letting $M\to\infty$ using monotone convergence, we may assume that $u$ is bounded. Recall also that $u$ is assumed to have compact support.

We write $u=\omega v$, multiply \eqref{eq:elreg} by $|v(x)|^p \omega(x)^p$ and integrate with respect to $x$. After symmetrizing with respect to $x$ and $y$ (recall that $k_\epsilon(x,y)=k_\epsilon(y,x)$) we obtain
\begin{align*}
 & \iint_{\R^N\times\R^N} \left(|v(x)|^p \omega(x) - |v(y)|^p \omega(y) \right) \left(\omega(x)-\omega(y)\right) \left|\omega(x)-\omega(y)\right|^{p-2} k_\epsilon(x,y) \,dx\,dy \notag \\
 & \qquad = \int_{\R^N} V_\epsilon(x) |u(x)|^p \,dx \ .
\end{align*}
We write this as
\begin{align}\label{eq:hardyreg}
\iint_{\R^N\times\R^N} \!\!\Phi_u(x,y) k_\epsilon(x,y) \,dx\,dy + \int_{\R^N} V_\epsilon |u|^p \,dx 
= \iint_{\R^N\times\R^N} \!\!|u(x)-u(y)|^p k_\epsilon(x,y) \,dx\,dy
\end{align}
where
\begin{align}\label{eq:phirem}
 \Phi_u(x,y) := & |\omega(x)v(x)-\omega(y)v(y)|^p  \notag\\
& - \left(\omega(x)|v(x)|^p - \omega(y) |v(y)|^p \right) 
\left(\omega(x)-\omega(y)\right) \left|\omega(x)-\omega(y)\right|^{p-2} \ .
\end{align}
We claim that $\Phi_u\geq 0$ pointwise. To see this, we may by symmetry assume that $\omega(x)\geq \omega(y)$. Putting $t=\omega(y)/\omega(x)$, $a= v(x)/v(y)$ and applying \eqref{eq:numbers} we deduce that $\Phi_u\geq 0$.

Now we pass to the limit $\epsilon\to 0$ in \eqref{eq:hardyreg}. Since $|u|^p$ is bounded with compact support and $V_\epsilon\to V$ weakly in $L_{1,\loc}$, the integral containing $V_\epsilon$ converges. The other two terms converge by dominated convergence since $0\leq k_\epsilon\leq k$, and we obtain
\begin{align}\label{eq:hardyequal}
\iint_{\R^N\times\R^N} \!\!\Phi_u(x,y) k(x,y) \,dx\,dy + \int_{\R^N} V |u|^p \,dx = E[u].
\end{align}
This implies the assertion since $\Phi_u\geq 0$.
\end{proof}

\begin{proof}[Proof of Proposition \ref{gsr}]
The proof is similar to that of Proposition \ref{hardy}, using \eqref{eq:numbersimproved} instead of \eqref{eq:numbers}. We omit the details.
\end{proof}

\begin{remark}\label{localass}
 Below we shall need a slight refinement of Propositions \ref{hardy} and \ref{gsr}. If in Assumption \ref{ass:el} the statement `$V_\epsilon\to V$ weakly in $L_{1,\loc}(\R^N)$' is replaced by the statement `$V_\epsilon\to V$ weakly in $L_{1,\loc}(\Omega)$ for an open set $\Omega\subset\R^N$', then \eqref{eq:hardy} and \eqref{eq:gsr} remain valid for $u$ with $\supp u \subset\Omega$. This is really what we have shown in the above proof.
\end{remark}

%%%%%%%%%%%%%%%%%%%%%%%%%%%%%%%%%%%%%%%%%%%%%%%%%%%%%%%%%%%%%%%%%%%%%%%%%%%%%%

\section{Proof of the sharp Hardy inequality}
\label{sec:proof}

Throughout this section we fix $N\geq 1$, $0<s<1$ and $p\neq N/s$ and abbreviate
$$
\alpha:=(N-ps)/p \, .
$$
We will deduce the sharp Hardy inequality \eqref{eq:main} using the general approach in the previous section with the choice
\begin{equation}\label{eq:abbreviate}
 \omega(x)= |x|^{-\alpha}\, ,
\quad
k(x,y) = |x-y|^{-N-ps}\, ,
\quad
V(x) = \mathcal C_{N,s,p} |x|^{-ps} \, .
\end{equation}

%%%%%%%%%%%%%%%%%%%%%%%%%%%%%%%%%%%%%%%%%%%%%%%%%%%%%%%%%%%%%%%%%%%%%%%%%%%%%%%%%

\subsection{The Euler-Lagrange equation}

We begin the proof of Theorem \ref{main} by verifying that $\omega$ solves the Euler-Lagrange equation associated with \eqref{eq:main}.

\begin{lemma}\label{eleq}
 One has uniformly for $x$ from compacts in $\R^N\setminus\{0\}$
\begin{equation}\label{eq:eleq}
 2 \lim_{\epsilon\to 0} \int_{\left|\left|x\right|-\left|y\right|\right|>\epsilon }
\left(\omega(x) -\omega(y)\right)
\left|\omega(x) - \omega(y) \right|^{p-2} k(x,y)\,dy
= \frac{\mathcal C_{N,s,p}}{|x|^{ps}} \ \omega(x)^{p-1}
\end{equation}
with $\mathcal C_{N,s,p}$ from \eqref{eq:mainconst}.
\end{lemma}

\begin{proof}
First note that it suffices to prove the convergence \eqref{eq:eleq} for a fixed $x\in\R^N\setminus\{0\}$, since the uniformity will then follow by a simple scaling argument. Now the integral on the left side of \eqref{eq:eleq} is absolutely convergent for any $\epsilon>0$ and after integrating out the angles it can be written as
\begin{equation}\label{eq:el1d}
 r^{-N+1} \int_{|\rho-r|>\epsilon} \frac{\sgn(\rho^\alpha-r^\alpha)}{|\rho-r|^{2-p(1-s)}} \, \phi(\rho,r) \,d\rho
\end{equation}
where $r=|x|$,
\begin{equation}\label{eq:phi}
 \phi(\rho,r)= \left|\frac{\rho^{-\alpha}-r^{-\alpha}}{r-\rho}\right|^{p-1} \times
\begin{cases}
\rho^{N-1} \left(1-\frac\rho r\right)^{1+ps} \Phi(\frac\rho r) , & \text{if}\ \rho<r,\\
r^{N-1} \left(1-\frac r\rho\right)^{1+ps} \Phi(\frac r\rho) , & \text{if}\ \rho>r,
\end{cases}
\end{equation}
and $\Phi=\Phi_{N,s,p}$ given in \eqref{eq:mainphi}. Since $p(1-s)>0$, the convergence of the integral in \eqref{eq:el1d} for $\epsilon\to 0$ will follow if we can show that $\phi(\rho,r)$ is Lipschitz continuous as a function of $\rho$ at $\rho=r$. For this we only need to prove that $(1-t)^{1+ps}\Phi(t)$ and its derivative remain bounded as $t\to 1-$.

For $N=1$ this is obvious and hence we restrict ourselves to the case $N\geq 2$ in the following. One can prove the desired property either directly using elementary estimates or, as we shall do here, deduce it from properties of special functions. According to \cite[(3.665)]{GR}
\begin{equation*}
 \Phi(t) = |\Sph^{N-2}| \ B(\tfrac{N-1}2,\tfrac12) \ F(\tfrac{N+ps}2,\tfrac{ps+2}2,\tfrac N2;t^2)
\end{equation*}
where $F(a,b,c;z)$ is a hypergeometric function. If $a+b-c>1$ then both $(1-z)^{a+b-c} F(a,b,c;z)$ and its derivative
\begin{equation*}
\frac d{dz} \left( (1-z)^{a+b-c} F(a,b,c;z) \right) = \frac{(c-a)(c-b)}c (1-z)^{a+b-c-1} F(a,b,c+1;z)
\end{equation*}
have a limit as $z\to 1-$; see \cite[Sec. 6.2.1, 6.8]{Lu}. Since $a+b-c=1+ps>1$ in our situation, one easily deduces that $(1-t)^{1+ps}\Phi(t)$ and its derivative have a limit as $t\to 1-$.

This argument gives \eqref{eq:eleq} with $\mathcal C_{N,s,p}$ replaced by the constant
\begin{equation}\label{eq:constant1d}
\mathcal C_{N,s,p}' :=
2 \lim_{\epsilon\to 0} \int_{|\rho-1|>\epsilon} \frac{\sgn(\rho^\alpha-1)}{|\rho-1|^{2-p(1-s)}} \, \phi(\rho,1) \,d\rho \, .
\end{equation}
To see that this constant coincides with \eqref{eq:mainconst}, we change variables $\rho\mapsto \rho^{-1}$ in the integral on $(1+\epsilon,\infty)$. Recalling the properties of $\phi$ we can pass to the limit $\epsilon\to0$ and obtain
\begin{align*}
\mathcal C_{N,s,p}' 
& = 2 (\sgn\alpha) \int_0^1 \left(\rho^{-p(1-s)} \phi(\rho^{-1},1) - \phi(\rho,1) \right) \frac{d\rho}{(1-\rho)^{2-p(1-s)}} \\
& = 2 \int_0^1 \rho^{ps-1} \left|1 - \rho^\alpha \right|^p \Phi(\rho) \, d\rho\,,
\end{align*}
which is \eqref{eq:mainconst} for $N\geq 2$. The proof for $N=1$ is similar.
\end{proof}

\begin{remark}
It is possible to express the sharp Hardy constant as an $N$-dimensional double integral
\begin{equation} \label{eq:constantn}
\mathcal C_{N,s,p} = \frac{|N-ps|}{p} \frac{2}{|\Sph^{N-1}|}
\iint_{\{|x|<1<|y|\}} \left| |x|^{-\tfrac{N-ps}p} - |y|^{-\tfrac{N-ps}p} \right|^{p-1} \frac{dx\,dy}{|x-y|^{N+ps}} \,.
\end{equation}
To see this in the case $N>ps$, we multiply the integral in \eqref{eq:eleq} by $\chi_B(x)$, the characteristic function of the unit ball $B\subset\R^N$, and integrate with respect to $x$. After symmetrizing with respect to the variables $x$, $y$ and passing to the limit $\epsilon\to 0$, we find
\begin{align*}
& \iint_{\R^N\times\R^N}
(\chi_B(x)-\chi_B(y)) \left(\omega(x) -\omega(y)\right)
\left|\omega(x) - \omega(y) \right|^{p-2} k(x,y) \,dx\,dy \\
& \qquad = \mathcal C_{N,s,p} \int_B \frac{\omega(x)^{p-1}}{|x|^{ps}} \, dx \,.
\end{align*}
Performing the integration on the right side yields \eqref{eq:constantn} for $N>ps$. In the case $N<ps$, we multiply \eqref{eq:eleq} by $1-\chi_B(x)$ and proceed similarly. 
\end{remark}

%%%%%%%%%%%%%%%%%%%%%%%%%%%%%%%%%%%%%%%%%%%%%%%%%%%%%%%%%%%%%%%%

\subsection{Proof of the Hardy inequality}

We apply the general approach in Section \ref{sec:general} with $k$, $\omega$, $V$ as in \eqref{eq:abbreviate} and
\begin{equation}
 k_\epsilon(x,y) :=
\begin{cases}
 k(x,y) & \text{if} \ \left||x|-|y|\right|>\epsilon\ ,\\
0 & \text{if} \ \left||x|-|y|\right|\leq \epsilon \ .
\end{cases}
\end{equation}
For simplicity, let ${\mathcal Q}=\dot W^s_p(\R^N)$ if $N>ps$, and ${\mathcal Q}=\dot W^s_p(\R^N\sm)$ if $N<ps$. 
In Lemma \ref{eleq} we have verified that the modification of Assumption \ref{ass:el} mentioned in Remark \ref{localass} is satisfied for $\Omega=\R^N\setminus\{0\}$. Inequalities \eqref{eq:main} and \eqref{eq:remainder} for $u\in C_0^\infty(\R^N\setminus\{0\})$ are an immediate consequence of Propositions \ref{hardy} and \ref{gsr}. By density they extend to the homogeneous Sobolev space ${\mathcal Q}$.

Next we shall prove that for $p>1$, inequality \eqref{eq:main} is
strict for all $0\not\equiv u\in{\mathcal Q}$. (Note that for
$p\geq 2$ this is an immediate consequence of \eqref{eq:remainder}.)
We start from identity \eqref{eq:hardyequal} which was proved for bounded functions $u$ with compact support in
$\R^N\setminus\{0\}$ and with $\int V|u|^p\,dx$ and $\iint |u(x)-u(y)|^p
k(x,y) \,dx\,dy$ finite. By a standard approximation argument, this
identity extends to any $u\in {\mathcal Q}$ with
$\Phi_u$ given by \eqref{eq:phirem} and $u=\omega v$. 

Assume that \eqref{eq:hardy} holds with equality for some $u\in {\mathcal Q}$, and hence also for $|u|$. Since $\Phi_{|u|}$ is non-negative and $k$ is strictly positive, it follows from \eqref{eq:hardyequal} that $\Phi_{|u|}\equiv 0$. Since $p>1$ this implies that $\omega(x)^{-p}|u(x)|^p$ is a constant (see Lemma \ref{numbers}), whence $u\equiv 0$. This proves that inequality \eqref{eq:hardy} is strict for any $0\not\equiv u\in {\mathcal Q}$ if $p>1$.

%%%%%%%%%%%%%%%%%%%%%%%%%%%%%%%%%%%%%%%%%%%%%%%%%%%%%%%%%%%%%%%%%%%%%%%%%%%%%%%%%%%%%%%%

\subsection{Sharpness of the constant}

To prove that the constant $\mathcal C_{N,s,p}$ in \eqref{eq:hardy} is optimal we first assume that $N>ps$ and use a family of trial functions $u_n\in\dot W^s_p(\R^N)$ which approximate the `virtual ground state' $\omega$. For any integer $n\in\N$ we divide $\R^N$ into three regions,
\begin{align*}
 I & := \{ x\in\R^N : \ 0\leq|x|< 1 \} \ , \\
M_n & := \{ x\in\R^N : \ 1\leq|x|< n \} \ , \\
O_n & := \{ x\in\R^N : \ |x|\geq n \} \ ,
\end{align*}
and define the functions
\begin{equation*}
 u_n(x) := 
\begin{cases}
 1 - n^{-\alpha} & \text{if } x\in I \ , \\
|x|^{-\alpha} - n^{-\alpha} & \text{if } x\in M_n \ , \\
0 & \text{if } x\in O_n \ .
\end{cases}
\end{equation*}
These functions belong to $W^1_p(\R^N)$ and hence also to $\dot{W}^s_p(\R^N)$. Similarly as in the proof of Proposition \ref{hardy} we integrate the right side of \eqref{eq:eleq} against $u_n(x)$ and symmetrize with respect to the variables. One easily shows that in the limit $\epsilon\to 0$ one obtains
\begin{align}\label{eq:trialeq}
 & \iint_{\R^N\times\R^N} (u_n(x) - u_n(y)) (\omega(x) - \omega(y)) | \omega(x) - \omega(y)|^{p-2} k(x,y)\,dx\,dy \notag \\
& \qquad = \mathcal C_{N,s,p} \int_{\R^N} \frac{u_n(x)\omega(x)^{p-1}}{|x|^{ps}} \,dx \ .
\end{align}
Here we use the same abbreviations as in \eqref{eq:abbreviate}. The left side of \eqref{eq:trialeq} can be rewritten as
\begin{equation*}
 \iint_{\R^N\times\R^N} |u_n(x) - u_n(y)|^p k(x,y)\,dx\,dy + 2 \, \mathcal R_0
\end{equation*}
with
\begin{align*}
 \mathcal R_0 := & \iint_{x\in I,\, y\in M_n} (1 - \omega(y))
\left( (\omega(x) - \omega(y))^{p-1} - (1 -\omega(y))^{p-1} \right) k(x,y) \,dx\,dy \\
& + \iint_{x\in M_n,\, y\in O_n} (\omega(x) - n^{-\alpha})
\left( (\omega(x) - \omega(y))^{p-1} - (\omega(x) - n^{-\alpha})^{p-1} \right) k(x,y) \,dx\,dy \\
& + \iint_{x\in I,\, y\in O_n} (1 - n^{-\alpha})
\left( (\omega(x) - \omega(y))^{p-1} - (1 - n^{-\alpha})^{p-1} \right) k(x,y) \,dx\,dy \ .
\end{align*}
It follows from the explicit form of $\omega(x)$ that the integrands in all three integrals are pointwise non-negative, hence
\begin{equation}\label{eq:r0}
 \mathcal R_0 \geq 0 \ .
\end{equation}
On the other hand, the right side of \eqref{eq:trialeq} divided by $\mathcal C_{N,s,p}$ can be rewritten as
\begin{equation*}
\int_{\R^N} \frac{u_n^p}{|x|^{ps}} \,dx + \mathcal R_1 + \mathcal R_2
\end{equation*}
with
\begin{align*}
 \mathcal R_1 & := \int_{I} \left(1 - n^{-\alpha}\right) \left( \omega(x)^{p-1} - \left(1 - n^{-\alpha}\right)^{p-1} \right) \frac{dx}{|x|^{ps}} \ , \\
\mathcal R_2 & := \int_{M_n} \left(\omega(x) - n^{-\alpha}\right) \left( \omega(x)^{p-1} - \left(\omega(x) - n^{-\alpha}\right)^{p-1} \right) \frac{dx}{|x|^{ps}} \ .
\end{align*}
Again both terms are non-negative and we shall show below that
\begin{equation}\label{eq:r12}
 \mathcal R_1 + \mathcal R_2 = \mathcal O(1)
\qquad \text{as } n\to\infty \ .
\end{equation}
Since obviously $\int u_n^p |x|^{-ps}\,dx\to\infty$ as $n\to\infty$ we conclude from \eqref{eq:r0} and \eqref{eq:r12} that
\begin{align*}
 & \frac{\iint_{\R^N\times\R^N} |u_n(x) - u_n(y)|^p k(x,y)\,dx\,dy}{\int_{\R^N} |u_n(x)|^p |x|^{-ps} \,dx} \\
& \qquad = \mathcal C_{N,s,p} \left(1 + \frac{\mathcal R_1 + \mathcal R_2}{ \int_{\R^N} |u_n(x)|^p |x|^{-ps} \,dx } \right) - \frac{2\mathcal R_0}{ \int_{\R^N} |u_n(x)|^p |x|^{-ps} \,dx } \\
& \qquad \leq \mathcal C_{N,s,p} \left(1 + o(1) \right)
\end{align*}
as $n\to\infty$. This shows that $\mathcal C_{N,s,p}$ is sharp.

It remains to prove \eqref{eq:r12}. Since the integrand in $\mathcal R_1$ is pointwise bounded by $ \omega(x)^{p-1}|x|^{-ps}=|x|^{\alpha-N}$ we find that
$\mathcal R_1 \leq \int_{|x|<1} |x|^{\alpha-N} dx< \infty$.  
To estimate $\mathcal R_2$ we use that $1-(1-t)^{p-1} \leq C_p t$ for $0\leq t\leq 1$ with $C_p=1$ for $1\leq p\leq 2$ and $C_p=p-1$ for $p>2$. Hence the integrand in $\mathcal R_2$ can be bounded according to
$$
\left(\omega(x) - n^{-\alpha}\right)
\left( \omega(x)^{p-1} - \left(\omega(x) - n^{-\alpha}\right)^{p-1} \right)
\leq C_p n^{-\alpha} \omega(x)^{p-1}  
$$
and therefore after extending the integral to all $|x|<n$ and scaling $x\mapsto x/n$ we obtain
$\mathcal R_2 \leq C_p \int_{|x|<1} |x|^{\alpha-N} dx < \infty$. 
This proves \eqref{eq:r12}.

The case $N>ps$ is treated similarly, using a sequence of trial functions of the form
$$
u_{n,m}(x) := \begin{cases} 0 & \text{if } |x|\leq 1/n \ , \\ |x|^{-\alpha} - n^{-\alpha} & \text{if } 1/n\leq |x| \leq 1\ , \\ (1-n^{-\alpha}) \chi(|x|/m) & \text{if } |x| \geq 1 \ , \end{cases}
$$
where $0\leq\chi\leq 1$ is a smooth, compactly supported function with
$\chi(t)=1$ for small $t$. After letting $m\to \infty$, the
calculation proceeds along the same lines as above.

%%%%%%%%%%%%%%%%%%%%%%%%%%%%%%%%%%%%%%%%%%%%%%%%%%%%%%%%%%%%%%%%%%%%%%%%%%%%%%%%%%%%%%%%

\subsection{The case $p=1$}

To conclude the proof of Theorem \ref{main} we need to characterize the minimizers in the case $p=1$. Actually, we present an alternative, simpler proof of inequality \eqref{eq:main} in this case based on a symmetrization argument.

Note that the right side of \eqref{eq:main} remains unchanged if $u$ is replaced by $|u|$, whereas the left side does not increase. Indeed, it strictly decreases unless $u$ is proportional to a non-negative function. Moreover, under symmetric decreasing rearrangement the left side of \eqref{eq:main} does not increase (see \cite{AL} and also Theorem \ref{rearrstrict}), whereas the right side does not decrease. Indeed, it strictly increases unless $|u|$ is symmetric decreasing (see \cite[Thm. 3.4]{LL}). This argument shows that any optimizer (provided it exists) will be proportional to a symmetric decreasing function. Below we show that \eqref{eq:main} holds with \emph{equality} for \emph{any} symmetric decreasing $u$. By the previous argument this provides an alternative proof of Theorem \ref{main} in the case $p=1$.

A symmetric decreasing function $u$ has a layer cake representation $u= \int_0^\infty \chi_t \,dt$ with $\chi_t$ the characteristic function of a ball centered at the origin with some radius $R(t)$. In this case the integral on the right side of \eqref{eq:main} equals
$$
\int_{\R^N} \frac{|u(x)|}{|x|^{s}} \,dx = \frac{|\Sph^{N-1}|}{N-s} \int_0^\infty R(t)^{N-s} \,dt \ ,
$$
and the integral on the left side equals
\begin{align*}
 \iint_{\R^N\times\R^N} \frac{|u(x)-u(y)|^p}{|x-y|^{N+ps} } \,dx\,dy
& = 2 \iint_{\{|x|< |y|\}} \frac{\left| \int (\chi_t(x)-\chi_t(y)) \,dt \right|}{|x-y|^{N+s} } \,dx\,dy \\
& = 2 \iiint_{\{|x|< R(t)<|y|\}} |x-y|^{-N-s} \,dx\,dy\,dt \\
& = 2 \iint_{\{|x|<1<|y|\}} |x-y|^{-N-s} \,dx\,dy \int_0^\infty R(t)^{N-s} \,dt.
\end{align*}
This shows that \eqref{eq:main} holds with equality for any symmetric decreasing function.

%%%%%%%%%%%%%%%%%%%%%%%%%%%%%%%%%%%%%%%%%%%%%%%%%%%%%%%%%%%%%%%%%%%%%%%%%%%%%%%%%%%%%%%%%%%%%%%%%%%%%%%%%%%%%%

\section{Sharp Sobolev embedding into Lorentz spaces}
\label{sec:embedding}

Let $1\leq q<\infty$, $1\leq r \leq \infty$ and recall that the Lorentz space $L_{q,r}(\R^N)$ consists of those measurable functions $u$ on $\R^N$ for which the following quasinorm is finite,
\begin{align*}
 \|u\|_{q,r} & := \left( q \int_0^\infty \mu_u(t)^{r/q} t^{r-1}\,dt \right)^{1/r}
\text{if}\ 1<r<\infty\,,
\quad
\|u\|_{q,\infty} := \sup_{t>0} \mu_u(t)^{1/q} t \, .
\end{align*}
Here $\mu_u(t):=\{x\in\R^N : |u(x)|>t\}$ denotes the distribution function of $u$. Note that $L_{q,q}(\R^N)=L_q(\R^N)$ and that one has strict inclusions $L_{q,r}(\R^N) \subset L_{q,s}(\R^N)$ for $r< s$. A classical result by Peetre \cite{P} states that the standard Sobolev embedding $\dot W^s_p(\R^N)\subset L_{p^*}(\R^N)$, $p^*=Np/(N-ps)$ for $N>ps$, can be improved to $\dot W^s_p(\R^N)\subset L_{p^*,p}(\R^N)$. Peetre's proof is based on interpolation and requires $p>1$. We refer to \cite{T} for more elementary interpolation arguments, including the case $p=1$.

Here we give a direct proof of this embedding which avoids interpolation. It is based on symmetrization and leads to sharp constants.

\begin{theorem}[Sharp Sobolev inequality]\label{embedding}
Let $N\geq 1$, $0<s<1$, $1\leq p<N/s$ and put $p^*=Np/(N-ps)$. Then $\dot W^s_p(\R^N)\subset L_{p^*,p}(\R^N)$ and
\begin{equation}\label{eq:embedding}
 \|u\|_{p^*,p} \leq \left(\frac{N}{|\Sph^{N-1}|}\right)^{s/N} \, \mathcal C_{N,s,p}^{-1/p} \left( \iint_{\R^N\times\R^N} \frac{|u(x)-u(y)|^p}{|x-y|^{N+ps} } \,dx\,dy \right)^{1/p}
\end{equation}
for any $u\in \dot W^s_p(\R^N)$ with $\mathcal C_{N,s,p}$ from \eqref{eq:mainconst}. This constant is optimal. For $p=1$ equality holds iff $u$ is proportional to a non-negative function $v$ such that the level sets $\{v>\tau\}$ are balls for a.e. $\tau$. For $p>1$ the inequality is strict for any $u\not\equiv 0$.
\end{theorem}

For $p=1$ and $u$ a characteristic function we obtain

 \begin{equation}\label{eq:isoperimetric}
|\Omega|^{(N-s)/N} \leq  \frac {2(N-s)}{ N \, \mathcal C_{N,s,1} } \left(\frac{N}{|\Sph^{N-1}|}\right)^{s/N} 
\iint_{\Omega\times\Omega^c} \frac{dx\,dy}{|x-y|^{N+s}} \, ,
\end{equation}
for any $\Omega\subset\R^N$ of finite measure, with equality iff $\Omega$ is a ball.
Moreover, using that
$$
\|u\|_{q,r} \leq \left(\frac qr\right)^{1/r} \left(\frac pq\right)^{1/p} \|u\|_{q,p}\, ,
\quad p<r\, ,
$$
(which is easily proved using the layer cake representation for $\mu_u^{p/q}$ and Minkowski's inequality) one obtains

\begin{corollary}\label{embeddingcor}
Let $N\geq 1$, $0<s<1$, $1\leq p<N/s$ and $p\leq r\leq\infty$. Put $p^*=Np/(N-ps)$. Then $\dot W^s_p(\R^N)\subset L_{p^*,r}(\R^N)$ and
\begin{equation}\label{eq:embeddingcor}
 \|u\|_{p^*,r} \leq \left(\frac {p^*}r\right)^{1/r} \left(\frac p{p^*}\right)^{1/p}
\left(\frac{N}{|\Sph^{N-1}|}\right)^{s/N} \, \mathcal C_{N,s,p}^{-1/p} 
\left( \iint_{\R^N\times\R^N} \frac{|u(x)-u(y)|^p}{|x-y|^{N+ps} } \,dx\,dy \right)^{1/p} \ .
\end{equation}
\end{corollary}

Setting $r=p^*$ in \eqref{eq:embeddingcor} we recover the standard Sobolev inequality \eqref{eq:sobemb}. Using the bound \eqref{eq:ms} on the constant, we recover the result \eqref{eq:bbm} by Maz'ya and Shaposhnikova. 

%%%%%%%%%%%%%%%%%%%%%%%%%%%%%%%%%%%%%%%%%%%%%%%%%%%%%%%%%%%%%%%%%%%%%%

The link between Theorem \ref{embedding} and the sharp Hardy inequality \eqref{eq:main} is

\begin{lemma}\label{symmdecr}
 Let $0<s\leq 1$ and $1\leq p<N/s$. Then for any non-negative, symmetric decreasing $u$ on $\R^N$
\begin{equation}\label{eq:symmdecr}
 \|u\|_{p^*,p} = \left( \frac N{|\Sph^{N-1}|} \right)^{s/N} 
\left(\int_{\R^N} \frac{u^{p}}{|x|^{ps}} \,dx \right)^{1/p} .
\end{equation}
\end{lemma}

\begin{proof}
 Introducing $w = u^p$ and $\mu=\mu_w$ we can rewrite the left side of \eqref{eq:symmdecr} as
\begin{equation*}
 \|u\|_{p^*,p}^p = \frac {p^*}p \int_0^\infty \mu(t)^{\tfrac p{p^*}} \,dt \, .
\end{equation*}
We write $w=\int_0^\infty \chi_t \,dt$ in its layer cake representation. Here $\chi_t$ is the characteristic function of $\{ x : w(x)> t \}$, which is a ball of radius $(N \mu(t)/ |\Sph^{N-1}|)^{1/N}$. Hence
\begin{equation*}
 \int_{\R^N} \frac{w}{|x|^{ps}} \,dx = \int_0^\infty \left( \int_{\R^N} \frac{\chi_t(x)}{|x|^{ps}} \,dx \right) \,dt
= \frac{N^{\tfrac{N-ps}N}}{N-ps} |\Sph^{N-1}|^{\tfrac{ps}N} \int_0^\infty \mu(t)^{\tfrac{N-ps}N} \,dt \, ,
\end{equation*}
proving \eqref{eq:symmdecr}.
\end{proof}

\begin{proof}[Proof of Theorem \ref{embedding}]
By symmetric decreasing rearrangement it suffices to prove \eqref{eq:embedding} for symmetric decreasing $u$ (see \cite{AL} and also Theorem \ref{rearrstrict}), for which it is an immediate consequence of Theorem \ref{main} and Lemma \ref{symmdecr}. The sharpness of the constant and the non-existence of optimizers for $p>1$ follows Theorem \ref{main}. For $p=1$ one uses the characterization of equality in the rearrangement inequality in Theorem~\ref{rearrstrict}.
\end{proof}

\begin{remark}\label{embeddingloc}
The `local' analog of \eqref{eq:embedding} for $s=1$ is 
\begin{equation}\label{eq:embeddingloc}
 \|u\|_{p^*,p} \leq
\left(\frac{N}{|\Sph^{N-1}|}\right)^{1/N} \, \frac p{N-p}
\left( \int_{\R^N} |\nabla u(x)|^p \,dx \right)^{1/p}
\end{equation}
for $N\geq 2$, $1\leq p<N$ and $p^*=Np/(N-p)$. It is 
due to \cite{ON,P}; the sharp constant in this case was found by Alvino \cite{Al}. 
Inequality \eqref{eq:embeddingloc} can be proved in the same way as Theorem \ref{embedding}, with the fractional Hardy inequality \eqref{eq:main} replaced by the classical Hardy inequality \eqref{eq:hardyclass}.
\end{remark}

%%%%%%%%%%%%%%%%%%%%%%%%%%%%%%%%%%%%%%%%%%%%%%%%%%%%%%%%%%%%%%%%%%%%%%%%%%%%%%%%%

\begin{appendix}

\section{A strict rearrangement inequality}
\label{sec:rearr}

Almgren and Lieb \cite{AL} have shown that the norm in $W^s_p(\R^N)$ does not increase under rearrangement. Since we have not found a characterization of the cases of equality in the literature, we include a proof. The special case $p=1$ has been used in the proof of Theorem \ref{embedding}.

\begin{theorem}\label{rearrstrict}
 Let $N\geq 1$, $0<s<1$, $1\leq p<N/s$ and $u\in\dot W^s_p(\R^N)$. Then
\begin{equation}\label{eq:rearrstrict}
 \iint_{\R^N\times\R^N} \frac{|u(x)-u(y)|^p}{|x-y|^{N+ps} } \,dx\,dy
\geq \iint_{\R^N\times\R^N} \frac{|u^*(x)-u^*(y)|^p}{|x-y|^{N+ps} } \,dx\,dy \ .
\end{equation}
If $p=1$, then equality holds iff $u$ is proportional to a non-negative function $v$ such that the level set $\{v>\tau\}$ is a ball for a.e. $\tau>0$. If $p>1$, then equality holds iff $u$ is proportional to a translate of a symmetric decreasing function.
\end{theorem}

Though we do not use the `only if' statement for $p>1$ in this paper, we have included it since we think it is interesting in its own right. It might be compared with the result in the `local case', namely, that if equality in $\int |\nabla u|^p \,dx \geq \int |\nabla u^*|^p \,dx$ is attained for a non-negative $u$, then the level sets of $u$ are balls, but $u$ is not necessarily a translate of a symmetric decreasing function; see \cite{BZ}.

We start by considering a slightly more general situation. For $J$ a non-negative, convex function on $\R$ with $J(0)=0$ and $k$ a non-negative function on $\R^N$, we let
$$
E[u] := \iint_{\R^N\times\R^N} J(u(x)-u(y)) k(x-y) \,dx\,dy \, .
$$

\begin{lemma}
 \label{rearrgeneral}
Let $J$ be a non-negative, convex function on $\R$ with $J(0)=0$ and let $k\in L_1(\R^N)$ be a symmetric decreasing function. Then for all non-negative measurable $u$ with $E[u]$ and $|\{u>\tau\}|$ finite for all $\tau>0$ one has
\begin{equation}
 \label{eq:rearrgeneral}
E[u] \geq E[ u^*] \, .
\end{equation}
If, in addition, $J$ is strictly convex and $k$ is strictly decreasing, then equality holds iff $u$ is a translate of a symmetric decreasing function. If $J(t)=|t|$, then equality holds iff the level sets $\{u>\tau\}$ are balls for a.e. $\tau>0$.
\end{lemma}

Inequality \eqref{eq:rearrgeneral} under the additional assumptions $J(t)=J(-t)$ and $\int J(u(x))\,dx<\infty$ is due to Almgren and Lieb \cite{AL}. The characterization of cases of equality seems to be new.

\begin{proof}
As in \cite[Thm. 3.5]{LL} we can write $J=J_++J_-$ with $J_+(t)=J(t)$ for $t\geq 0$ and $J_+(t)=0$ for $t<0$. We decompose $E=E_++E_-$ accordingly. Below we prove the assertion of the lemma with $E$ replaced by $E_+$. The assertion for $E_-$ (and hence for the original $E$) follows by exchanging the roles of $x$ and $y$ and replacing $J(t)$ by $J(-t)$. Note that this argument yields a characterization of cases of equality under the weaker assumption that $J$ is strictly convex on either $\R_+$ or $\R_-$.

\emph{Step 1.} We first prove the assertion under the additional assumption that $u$ is bounded. Since $J_+$ is convex it has a right derivative $J_+'$, which is non-negative and non-decreasing. Writing $J_+(t)=\int_0^t J_+'(\tau)\,d\tau$ one finds
\begin{equation*}
 J_+(u(x)-u(y)) = \int_0^\infty J_+'(u(x)-\tau) \chi_{\{u\leq \tau\}}(y) \,d\tau \, ,
\end{equation*}
and hence by Fubini
\begin{equation}\label{eq:rearrgenrepr}
 E_+[u] = \int_0^\infty e^+_\tau[u] \,d\tau
\end{equation}
where
\begin{equation}
 e^+_\tau[u] :=  \iint_{\R^N\times\R^N} J_+'(u(x)-\tau) k(x-y) \chi_{\{u\leq \tau\}}(y) \,dx\,dy \ .
\end{equation}
Since $u$ is bounded and $|\{u>\tau\}|<\infty$ one has $\int_{\R^N} J_+'(u(x)-\tau) \,dx<\infty$. Writing $\chi_{\{u\leq \tau\}}=1-\chi_{\{u> \tau\}}$ we obtain
\begin{equation}\label{eq:rearrgenrepr2}
 e^+_\tau[u] = \|k\|_1 \int_{\R^N} J_+'(u(x)-\tau) \,dx -
\iint_{\R^N\times\R^N} J_+'(u(x)-\tau) k(x-y) \chi_{\{u> \tau\}}(y) \,dx\,dy \, .
\end{equation}
The first integral on the right side of \eqref{eq:rearrgenrepr2} does not change under rearrangement. Moreover, we note that $\left(J_+'(u-\tau)\right)^*=J_+'(u^*-\tau)$. By Riesz's rearrangement inequality, the double integral on the right side of \eqref{eq:rearrgenrepr2} does not decrease under rearrangement, proving $e^+_\tau[u]\geq e^+_\tau[u^*]$ and hence $E_+[u]\geq E_+[u^*]$.

To characterize the cases of equality assume that $k$ is strictly decreasing and $E_+[u]= E_+[u^*]$ for some bounded $u$. Then by \eqref{eq:rearrgenrepr2} $e^+_\tau[u]=e^+_\tau[u^*]$ for a.e. $\tau$, and by Lieb's strict rearrangement inequality \cite{L} for a.e. $\tau>0$ there is an $a_\tau\in\R^N$ such that $\chi_{\{u< \tau\}}(x)=\chi_{\{u^*< \tau\}}(x-a_\tau)$ and 
\begin{equation}
J_\pm'(u(x)-\tau) = J_\pm'(u^*(x-a_\tau)-\tau)
\end{equation}
for a.e. $x$. If $J_+(t)=t_+$ for all $t$, this means that $\{u>\tau\}$ is a ball for a.e. $\tau>0$. Now assume that $J_+$ is strictly convex on $\R_+$. Then $J_+'$ is strictly increasing on $\R_+$ and we conclude that $(u(x)-\tau)_+ = (u^*(x-a_\tau)-\tau)_+$ for a.e. $\tau$ and $x$. This is easily seen to imply that $a_\tau$ is independent of $\tau$, and hence $u$ is a translate of a symmetric decreasing function.

\emph{Step 2.} Now we remove the assumption that $u$ is bounded, that is, we claim that \eqref{eq:rearrgeneral} holds for any non-negative $u$ with $E[u]$ and $|\{u>\tau\}|$ finite for all $\tau$. To see this, replace $u$ by $u_M=\min\{u,M\}$ and note that $(u_M)^*=(u^*)_M=:u_M^*$ and $E[u_M]\leq E[u]$. By monotone convergence the claim follows easily from $E[u_M]\geq E[u_M^*]$.

\emph{Step 3.} Finally, we characterize the cases of equality for general $u$. Assume that $k$ is strictly decreasing and $E_+[u]=E_+[u^*]$ for some non-negative $u$ with $E[u]$ and $|\{u>\tau\}|$ finite for all $\tau$. For any $M>0$ we decompose
\begin{equation*}
 u=u_M + v_M\, ,
\quad u_M:=\min\{u,M\}\,,
\end{equation*}
and find
\begin{equation}\label{eq:rearrgenrepr3}
 E_+[u] = E_+[u_M] + E_+[v_M] + \iint_{\R^N\times\R^N} F_M(v_M(x),u_M(y)) k(x-y) \,dx\,dy
\end{equation}
with
\begin{equation*}
 F_M(v,u) := J_+(v+M-u) - J_+(v) - J_+(M-u) \, .
\end{equation*}
Note that since $J_+$ is convex with $J_+(0)=0$, one has $F_M(v,u)\geq 0$ for $0\leq u \leq M$ and $v\geq 0$. Hence all three terms on the right side of \eqref{eq:rearrgenrepr3} are non-negative and finite. Note that replacing $u$ by $u^*$ amounts to replacing $u_M$ and $v_M$ by $u_M^*$ and $v_M^*$, respectively. Below we shall prove that the double integral in \eqref{eq:rearrgenrepr3} does not increase if both $u_M$ and $v_M$ are replaced by $u_M^*$ and $v_M^*$. Moreover, by Step 2, $E^+[v_M]\geq E^+[v_M^*]$. Hence if $E^+[u]=E^+[u^*]$, then $E^+[u_M]=E^+[u_M^*]$ for all $M>0$. 
Using the characterization from Step 1 one easily concludes that $u$ is of the form stated in the lemma.

It suffices to prove that the double integral in \eqref{eq:rearrgenrepr3} does not increase under rearrangement. Since $J'_+$ is increasing, we have $J_+'(t)=\int_0^t d\mu(\tau)$ for a non-negative measure $\mu$. Hence $J_+(t)=\int_0^\infty (t-\tau)_+ \,d\mu(\tau)$ and
\begin{equation*}
F_M(v,u) = \int_0^\infty\!\! f_{M,\tau}(v,u) \,d\mu(\tau) \, ,
\
f_{M,\tau}(v,u) := (v+M-u-\tau)_+ - (v-\tau)_+ - (M-u-\tau)_+ \, .
\end{equation*}
Since the integrand is non-negative for $0\leq u \leq M$ and $v\geq 0$, it suffices to prove that for all $\tau$ the double integral
\begin{equation*}
\iint_{\R^N\times\R^N} f_{M,\tau}(v_M(x),u_M(y)) k(x-y) \,dx\,dy
\end{equation*}
does not increase under rearrangement. We decompose further $f_{M,\tau} = f_{M,\tau}^{(1)} - f_{M,\tau}^{(2)}$ where
$f_{M,\tau}^{(1)}(v) := v - (v-\tau)_+$ and
$$
f_{M,\tau}^{(2)}(v,u) := v - (v+M-u-\tau)_+ + (M-u-\tau)_+ = \min\left\{v,(u-M+\tau)_+\right\} \, .
$$
Since $f_{M,\tau}^{(1)}$ is bounded and the support of $v_M$ has finite measure, the integral
$$
\iint_{\R^N\times\R^N} f_{M,\tau}^{(1)}(v_M(x)) k(x-y) \,dx\,dy
= \|k\|_1 \int_{\R^N} f_{M,\tau}^{(1)}(v_M(x)) \,dx
$$
is finite and invariant under rearrangement of $v_M$. Finally, by Fubini we can write
\begin{align*}
& \iint_{\R^N\times\R^N} f_{M,\tau}^{(2)}(v_M(x),u_M(y)) k(x-y) \,dx\,dy \\
& \qquad = \int_0^\infty \left( \iint_{\R^N\times\R^N} \chi_{\{v_M>t\}}(x) k(x-y) \chi_{\{(u_M-M+\tau)_+>t\}}(y) \,dx\,dy \right) \,dt \, .
\end{align*}
By Riesz's rearrangement inequality, this does not decrease under rearrangement, completing the proof.
\end{proof}

\begin{proof}[Proof of Theorem \ref{rearrstrict}]
First note that $|u(x)-u(y)|\geq \left||u(x)|-|u(y)|\right|$, and that equality for all $x, y$ holds iff $u$ is proportional to a non-negative function. Hence we can restrict ourselves to non-negative functions. Writing as in \cite{AL}
\begin{equation}\label{eq:gaussian}
 \iint_{\R^N\times\R^N} \frac{|u(x)-u(y)|^p}{|x-y|^{N+ps} } \,dx\,dy 
= \frac 1{\Gamma((N+ps)/2)} \int_0^\infty I_\alpha[u] \alpha^{(N+ps)/2-1}\,d\alpha 
\end{equation}
with
\begin{equation*}
 I_\alpha[u]:=\iint_{\R^N\times\R^N} |u(x)-u(y)|^p e^{-\alpha |x-y|^2} \,dx\,dy \ ,
\end{equation*}
the assertion follows from Lemma \ref{rearrgeneral}.
\end{proof}

\end{appendix}

%%%%%%%%%%%%%%%%%%%%%%%%%%%%%%%%%%%%%%%%%%%%%%%%%%%%%%%%%%%%%%%%%%%%%%%%%%%%%%%%%

\bibliographystyle{amsalpha}

\end{document}